\documentclass[11pt]{amsart}
\usepackage{times}
\usepackage[T1]{fontenc}
\usepackage{amssymb, amsthm, amsmath}
\usepackage{mathrsfs}
\usepackage{mathtools}

\usepackage[active]{srcltx}

\usepackage{todonotes}

\usepackage{setspace}
\usepackage{geometry}

\usepackage{hyperref}

\DeclareMathOperator{\eq}{eq}

\DeclareMathOperator{\rv}{rv}

\newtheorem{introtheorem}{Theorem}

\newtheorem{theorem}{Theorem}[section]

\newtheorem{claim}{Claim}[theorem]

\newtheorem{corollary}[theorem]{Corollary}

\newtheorem{fact}[theorem]{Fact}
\newtheorem{lemma}[theorem]{Lemma}

\newtheorem{proposition}[theorem]{Proposition}

\newtheorem*{gen-dif}{\fbox{{\large A}} \hypertarget{Agen-dif}{Gen-Dif}}

\newtheorem*{min-balln}{\fbox{{\large A}} \hypertarget{Amin-ball}{Cballs}}

\theoremstyle{definition}
\newtheorem{definition}[theorem]{Definition}
\newtheorem{example}[theorem]{Example}
\newtheorem{remark}[theorem]{Remark}
\newtheorem{question}[theorem]{Question}
\newtheorem{notation}[theorem]{Notation}

\newcommand{\Nn}{{\mathbb{N}}}

\newcommand{\m}{\textbf{m}}
\newcommand{\bk}{\textbf{k}}

\newcommand{\CD}{{\mathcal D}}
\newcommand{\CL}{{\mathcal L}}
\newcommand{\CK}{{\mathcal K}}

\newcommand{\CM}{{\mathcal M}}

\newcommand{\CC}{{\mathcal C}}

\newcommand{\CO}{{\mathcal O}}

\newcommand{\0}{\emptyset}

\renewcommand{\phi}{\varphi}

\def\sub{\subseteq}

\newenvironment{claimproof}[1][\proofname]
{%
	\proof[#1]%
}
{%
	\endproof%
}

\title{A note on uniform finiteness in weakly o-minimal expansions of groups}

\author{Yatir Halevi}
\address{Faculty of Mathematics\\Technion - Israel Institute of Technology\\Haifa\\ Israel.}
 \email{yatirh@math.technion.ac.il}

\author{Assaf Hasson}
\address{Department of Mathematics, Ben Gurion University of the Negev, Be'er-Sheva, Israel}
\email{hassonas@math.bgu.ac.il}

\author{Ya'acov Peterzil}
\address{Faculty of Natural Sciences, Department of Mathematics, University of Haifa, Haifa, Israel}
\email{kobi@math.haifa.ac.il}

\date{\today}
\begin{document}

 \thanks{The first and second authors were partially supported by ISF grant No. 555/21. The third author was supported by ISF grant No. 290/19.}
    
	\maketitle
     \begin{abstract}
         For an $\aleph_0$-saturated weakly o-minimal expansion of an ordered group $\CM$, it is shown that $\CM^{\eq}$ has uniform finiteness if and only if the collection of definable convex subgroups of $\CM$ has uniform finiteness. If $\CM$ expands an ordered field, then considering definable convex valuation subrings is sufficient. The results use a criterion of Johnson for uniform finiteness in $\CM^{\eq}$, \cite{JohnUF}.
        
        
        In addition, it is shown that uniform finiteness in $\CM^{\eq}$ may fail for weakly o-minimal expansions of fields.

     \end{abstract}
    \section{Introduction}

A definable (or interpretable) set $X$ in a first-order structure $\CM$ has \emph{uniform finiteness} if for any definable family of subsets $\{D_t\subseteq X: t\in T\}$ there is a natural number $n\in \mathbb{N}$ for which $|D_t|=\infty\iff |D_t|>n$. A structure $\CM$ has uniform finiteness if all its definable subsets have uniform finiteness.  Many well-behaved structures have uniform finiteness; but not all, for example $(\mathbb{Z},+,<)$.

Uniform finiteness in $\CM^{\eq}$ is a different matter (see Definition \ref{D: uniform} for a precise meaning of the term): The p-adic field $\mathbb{Q}_p$ has uniform finiteness, but its interpretable value group does not.

In the present note we study uniform finiteness in weakly o-minimal expansions of groups. Recall that a linearly ordered structure $\CM$ is \textit{weakly o-minimal } if every definable subset of $M$ is a finite union of convex sets. Throughout we will be working only with $\aleph_0$-saturated weakly o-minimal densely ordered  structures, so we assume that the theory  of the structure is  weakly o-minimal. Weakly o-minimal group are dense and so, essentially by definition, have uniform finiteness. Hence, we focus on uniform  finiteness in imaginary sorts. Our main results are: 

\begin{introtheorem}[Theorems \ref{T: convex subgroups is enough} and \ref{T:Main, UF iff conv val rings}]
    Let $\CM$ be an $\aleph_0$-saturated weakly o-minimal expansion of an ordered group. Then 
    \begin{enumerate}
        \item $\CM^{eq}$ has uniform finiteness if and only if the collection of definable convex\footnote{By \cite[Lemma 5.2]{MacMaSt} the convexity assumption is redundant. We keep it here and throughout for improved readability.}  subgroups of $\CM$ has uniform finiteness. 
        \item If, moreover,  $\CM$ is an expansion of an ordered field, then $\CM^{eq}$ has uniform finiteness if and only if the collection of  definable convex valuation subrings of $\CM$ has uniform finiteness.  
    \end{enumerate}
\end{introtheorem}
Our proof builds on a general criterion of Johnson, \cite[Theorem 2.3]{JohnUF}, for a structure $\CM$ to have uniform finiteness in all imaginary sorts. Specifically, Johnson proves that $\CM^{\eq}$ has uniform finiteness if and only if every definable set of \emph{unary imaginaries} has uniform finiteness. A definable set $X$ in $\CM^{\eq}$ is a {\em set of unary imaginaries} if $X$ is a collection (of codes)  of pairwise distinct subsets of $M$.

Using this result, we also prove an Ax-Kochen-Ershov style transfer result, namely that for a weakly o-minimal ordered field $\CM$, $\CM^{\eq}$ has uniform finiteness if and only if for some (equivalently, any) definable valuation $v$ (possibly trivial) the collections of all $M$-definable convex subgroups of $\Gamma_v$ (the value group) and of the additive group of $\bk_v$ (the residue field) have uniform finiteness (see Corollary \ref{C: UF for wom valued fields} for a more precise statement). This result extends our proof from \cite{HaHaPeGps}, asserting that power bounded T-convex expansions of o-minimal fields have uniform finiteness. 

To show that the above results are not vacuous, we give in the appendix an example of a weakly o-minimal expansion of a real closed valued field $\CK$ where uniform finiteness fails in $\CK^{\eq}$. It is a consequence of the following theorem together with  \cite[Example A.6]{HaHaPeSs}.

\begin{introtheorem}[Theorem \ref{T:appendix theorem}]
   Let $\CK$ be a pure real closed valued field. Let $\tilde k$, $\tilde \Gamma$ be weakly o-minimal expansions of the $\CK$-induced structure on the residue field and value group sorts. Then the resulting expansion  $\widetilde{\CK}$ of $\CK$ is weakly o-minimal. 
\end{introtheorem}

	\noindent{\bf Acknowledgment} 
	We would like to thank the anonymous referee for a careful reading and useful comments. 

    \noindent{\bf AI statement} No AI was used during the preparation of this paper. 
	
\subsection*{Conventions} Throughout, given a structure $\CM$, unless specified otherwise, by a definable set we mean a definable set in $\CM^{\eq}$ with parameters.

\section{Definable cuts in ordered abelian groups}

The first set of definitions and facts are mostly known and taken, for example, from  \cite{FVKNar,Tressl}. We repeat them for the reader's convenience, and in order to fix some notation.

Let $(M,<,+)$ be a linearly ordered abelian group. By {\em a cut in $M$} we mean $\mathcal C=(C_1,C_2)$, where $C_1$ is a non-trivial initial segment of $M$ and $C_2$ is its complement. 
Abusing terminology, we will identify the cut $\CC$ and the corresponding initial segment $C_1$.

\begin{definition}\label{D: Stab}
Given a cut $\CC=(C_1,C_2)$ in $M$, the stabilizer group of $\CC$, also known as \emph{the invariance group of $\CC$} is  $G(C_1)=G(\CC)=\{\delta\in M: \delta+C_1=C_1\}$. 
\end{definition}

The invariance group $G(\CC)$ is a convex (possibly trivial) subgroup of $(M,+)$. The following is standard:

\begin{definition}\label{D: cuts}
	Let $\CC=(C_1,C_2)$ be a cut. 
	\begin{enumerate}
		\item  $\CC$ is {\em valuational} if $G(\CC)\neq \{0\}$ and {\em non-valuational}, otherwise.
		\item A non-valuational cut in $M$ is \emph{rational} if it is realized (i.e., $C_1$ has a supremum) in $M$, and \emph{irrational} otherwise. 
		\item We say that a non-empty  initial segment  of $M$ is (non-)valuational or (ir)rational if its corresponding cut is such.
	\end{enumerate}	
\end{definition}

   If $\CC$ is a valuational cut in $(M,+,\cdot,<,\dots)$ a weakly o-minimal expansion of a real closed field, then for $G=G(\CC)$ the additive stabilizer of the cut $\CC$, $\{x: xG\sub G\}$ is a non-trivial valuation ring, \cite[Lemma 6.2]{MacMaSt}. 

For a convex set $D$ we write, as usual, $x<D$ to mean $x<c$ for all $c\in D$.  We also write $x\le D$ for $x<D\vee x\in D$. With this notation the following is immediate: 

\begin{lemma}\label{L:all cuts}
	Let $\CC=(C_1,C_2)$ be  a cut and  $H:=G(\CC)$. Then exactly one of the following holds:
\begin{enumerate}
    \item[(i)] $\CC$ is  irrational and non-valuational.
    \item[(ii)] There is a coset $S$ of $H$, possibly a singleton, such that $C_1=\{x\in M:x\,\Box\, S\}$, with $\Box\in \{<,\leq\}$. 
    \item[(iii)] $\CC$ is valuational  and for $\pi:M\to M/H$ the natural quotient map, $\pi(C_1)$ is an irrational non-valuational initial segment. 
\end{enumerate}	
\end{lemma}

\begin{definition}
	A valuational cut $\CC=(C_1,C_2)$ is called \textit{of type (iii)}, if $\pi(C_1)$ is irrational non-valuational in the quotient group $M/G(C_1)$.
\end{definition}

\begin{remark}
    The literature contains finer classification of cuts, for example see \cite[Figure 1]{FVKNar}. In the terminology of \cite{FVKNar}, type (iii) cuts are non-ball cuts with non-trivial invariance group. In the terminology of \cite{Tressl}, type (iii) cuts are cuts with sign $0$ and non-trivial invariance group.
\end{remark}

\begin{example}
    Consider $M=(\mathbb{Q},+)^2$ with the lexicographic ordering. The cut $\mathcal{C}$ induced by $(\sqrt{2},0)$ is of type (iii). Indeed, $G(\mathcal{C})=0\times\mathbb{Q}$ and the induced cut on $M/G(\mathcal{C})$ is irrational (see \cite[Ex3]{FVKNar}).

    An example of a cut of type (iii) in a weakly o-minimal expansion of a field can be found in \cite[Example 3.11(C)]{Tressl-modcomp}.
\end{example}

Since every definable subset of a weakly o-minimal structure is a finite boolean combination of definable initial segments we conclude:

\begin{lemma}\label{L: definable sets in wom}
	If $\CM=(M,+,<,\dots)$ is a weakly o-minimal expansion of an ordered group then every definable subset of $M$ is a boolean combination of definable initial segments as in Lemma \ref{L:all cuts} (i)-(iii).
\end{lemma}

The next lemma will not be used in the sequel, but may be of independent interest. It is a generalization of \cite[Lemma 3.1]{HaHaPeSs}. For a convexly valued ordered field $\CK$ we say that the value group $\Gamma$ (resp. the residue field $\bk$) is \emph{definably complete} if any $\CK$-bounded definable subset of $\Gamma$ (resp. $\bk$) has a supremum. With this terminology we have: 

\begin{lemma}\label{L: o-minimal group and res}
	Assume that $(\CK,+,\cdot,<,\CO,\dots)$ is an expansion of a convexly valued ordered field, such that 
	$\bk$ and $\Gamma$ are definably complete. Then,
   \begin{enumerate}
       \item 
  
	Every non-trivial definable convex subgroup of $(K,+)$ is a ball.  
  \item   
    For any definable valuational cut $\CC=(C_1,C_2)$ in $K$ of type (iii), $G(\CC)$ is a closed ball.
     \end{enumerate}
\end{lemma} 
\begin{proof}
   (1) Let $H$ be a non-trivial definable convex subgroup of $(K,+)$. Consider balls of the form $B_{\geq r}(0)$ and $B_{>r}(0)$ for $r\in \Gamma$. As $H$ is a convex subgroup it is comparable by inclusion to every ball (open or closed).  If $B_{\geq r}(0)\subseteq H$ for every $r\in \Gamma$ then $H=K$, so the set  $S:=\{r:B_{\geq r}(0)\sub H\}$ must be bounded below. By definable completeness of $\Gamma$, it has an infimum $r_0:=\inf S$.

    If $r_0$ is not a minimum of $S$, i.e. $r_0\notin S$, then $B_{>r_0}(0)\subsetneq H\subseteq B_{\geq r_0}(0)$. As $B_{\ge r_0}(0)/ B_{> r_0}(0)\cong \bk$, $H$ is mapped via this isomorphism onto a proper definable convex subgroup of $(\bk, +)$, which by the definable completeness of $\bk$ must be trivial, hence   $H=B_{> r_0}(0)$.

    Assume that $r_0\in S$. If $r_0$ has an immediate predecessor $s$ then $B_{>s}(0)=B_{\geq r_0}(0)$ and $B_{>s}(0)\subseteq H\subsetneq B_{\geq s}(0)$ and we conclude as in the previous case.  So assume that $r_0\in S$ and that it has no immediate predecessor. It is easy to see that $H=B_{\geq r_0}(0)$ a closed ball.

 Thus, $H$ is either a closed or an open ball.

(2) 	Now assume that $C_1$ is a valuational initial segment of type (iii). So $H=G(C_1)$ is nontrivial convex and $\pi(C_1)\subseteq K/H$ is irrational and non-valuational. After rescaling we may assume that $H$ is a ball of radius $0$. We claim that $H=\CO$. 
    
    Indeed, if $H=\m$, the maximal ideal, then there exists $z\in \CO^\times$ which does not lie in $H$, i.e. there are elements $x\in C_1$ and $y\notin C_1$ such that $v(y-x)=0$. So there is a closed ball of radius $0$  containing both $x$ and $y$. Translating, we may assume that this closed ball is $\CO$. Thus, $\pi:K\to K/\m$ sends $\CO$ onto $\bk$ and sends $C_1\cap \CO$ to a nontrivial cut in $\bk$. By definable completeness of $\bk$, $\pi(C_1)$ has a supremum, contradicting the assumption that it is irrational.
\end{proof}    

Keeping the notation of the lemma, if $\CK$ is weakly o-minimal, so are $\Gamma$ and $\bk$. In this setting, $\Gamma$ and $\bk$ are definably complete if and only if they are o-minimal (and therefore stably embedded, see \cite{HaOn2a}).

 Below, we say that a cut $\CC$ is called {\em a ball-cut} if there exists a ball $B\sub K$ such that $C=\{x\in K: x\, \square\, B\}$, where $\square \in \{<,\leq \}$.
\footnote{In \cite{FVKNar} the notion of ball-cut is defined similarly, allowing $B$ to be any convex subgroup of $(K,+)$.}
 By what we have seen thus far, we may conclude:
\begin{corollary}
	If $\CK$ is a weakly o-minimal expansion of a valued field and both $\Gamma$ and $\bk$ are o-minimal then every definable subset of $K$ is a boolean combination of non-valuational cuts, ball-cuts and valuational cuts of type (iii), with stabilizers a closed ball.
\end{corollary}

\section{Uniform finiteness in $\CM^{\eq}$}
 In this section we prove the main theorem stated in the introduction. For that, we first refine the definition of uniform finiteness from the introduction; it is based on Johnson's  \cite{JohnUF}. 

\begin{definition}\label{D: uniform}
    Let $\CM$ be a sufficiently saturated structure.
    \begin{enumerate}
        \item A definable (by our convention, possibly interpretable) set $X$ \emph{has uniform finiteness} if for every definable family $\{Y_t\subseteq X:t\in T\}$ there is a natural number $n\in \mathbb{N}$ such that for all $t\in T$: $|Y_t|=\infty \iff |Y_t|>n.$
        \item Let $\CD$ be a (not necessarily definable) collection of definable subsets of $\bigcup_n M^n$. A definable set $X$ is \emph{a set of imaginaries from $\CD$} if there is a definable relation $R\subseteq X\times M^n$ such that
        \begin{enumerate}
            \item $x\mapsto R_x:=\{m\in M^n:(x,m)\in R\}$ is injective, and
            \item for every $x\in X$, $R_x\in \CD$.
    \end{enumerate}
    \item 
          We say that \emph{$\CD$ has uniform finiteness} 
         if any definable set of imaginaries from $\CD$ has uniform finiteness.
         \end{enumerate}

\end{definition}

With this definition, Johnson's theorem \cite[Theorem 2.3]{JohnUF} states  that $\CM^{\eq}$ has uniform finiteness if and only if the collection of all definable subsets of $M=M^1$ has uniform finiteness. When $\CM$ is weakly o-minimal we can reduce to an even simpler collection of sets, as discussed below.  \\

\textbf{From now on we fix  $\CM=(M,+,<,\dots)$, a sufficiently saturated expansion of a weakly o-minimal ordered group}. 

\begin{lemma}\label{L: initial segments}
   If the collection of initial segments  of $M$ has uniform finiteness then $\CM^{\eq}$ has uniform finiteness.
\end{lemma}
\begin{proof}
By Johnson's theorem, it suffices to show that the collection of definable subsets of $M$ has uniform finiteness. To each definable $X\sub M$ we let $C(X)$ denote the (finite) collection  of 
initial  segments  determined by its convex components (each convex set determines two initial segments).  For a set of unary codes $Y$ we let $C(Y):=\bigcup\limits_{X\in Y} C(X)$. 

Let us verify that $Y$ is finite if and only if $C(Y)$ is finite. Indeed, a set $C(Y)$ of $n$ initial segments gives rise to at most $f(n)=2^{2^n}$ boolean combinations of these segments, thus there are at most $f(n)$-many definable sets $X\sub M$ whose code belongs to $Y$, namely $|Y|\leq f(n)$. The other direction is obvious. 

This implies that uniform finiteness of the collection of initial segments implies uniform finiteness of the collection of definable subsets of $M$, as claimed. 
\end{proof}

Our goal here is to establish which  definable sets in $\CM^{eq}$ necessarily have uniform finiteness. 
An obvious obstruction  is the existence of a definable infinite discrete linear order, where by a discrete linear order we mean one where every non-extremal element has an immediate successor and an immediate predecessor.   We start by noting that this is, in fact, the only obstacle to $\CM^{eq}$ having uniform finiteness. 

\begin{lemma}\label{L: eliminating infty equiv nondiscrete}
	$\CM^{\eq}$ does not have uniform finiteness if and only if there exists a definable convex equivalence relation $E$ on $M$ such that $(M/E,<)$ is infinite and discrete.
\end{lemma}
\begin{proof}
	Infinite discrete linear orders fail uniform finiteness (as witnessed by the family of intervals). So interpretable discrete linear orders contradict uniform finiteness for $\CM^{eq}$. 
    
Conversely, if $\CM^{\eq}$ does not have uniform finiteness, we prove that $\CM$ interprets an infinite discrete linear order.  By Lemma \ref{L: initial segments}, failure of uniform finiteness appears in some definable family of sets of initial segments of $M$. The initial segments are linearly ordered by inclusion. If $\{Y_t\}_{t\in T}$ is a family of codes such that each $Y_t$ consists of codes of initial segments, then each $Y_t$ is linearly ordered, uniformly in $t$. In particular, if $Y_t$ is finite it is discrete with respect to its natural ordering. 

Thus, if $\{Y_t\}_{t\in T}$ fails uniform finiteness, by compactness and saturation, there is some $t_0\in T$ such that $Y_{t_0}$ is infinite and discretely ordered. Given such a set $Y_{t_0}$, denote for $x\in M$ the set $C(x):=\{y\in Y_{t_0}: x\in y\}$.  The  equivalence  relation $E(x_1,x_2)$ on $M$ given by $C(x_1)=C(x_2)$ has the desired properties. 
\end{proof}

\subsection{Uniform finiteness in groups}

Lemma \ref{L: initial segments} reduces uniform finiteness in $\CM^{\eq}$ to uniform finiteness of the collection of all initial segments of $M$. The aim of this section is to reduce it further to the collection of all convex subgroups. We first prove this important observation.

\begin{lemma}\label{L: eliminating non-valuational}
The collection of all non-valuational (possibly rational) definable cuts has uniform finiteness. 
\end{lemma}
\begin{proof} 
Assume that $\{Y_t:t\in T\}$ is a definable family of sets of unary codes, each consisting of (codes of) non-valuational initial segments.  For each $t\in T$ denote $I_t:=\{y\in Y_t: y \text{ is a code of a rational cut}\}$ and $J_t$ its complement in $Y_t$. By uniform finiteness in $M$ there is $n\in \Nn$ such that $|I_t|<n$ if and only if $|I_t|<\infty$. So we may assume that $Y_t=J_t$ for all $t$.

As in the proof of Lemma \ref{L: eliminating infty equiv nondiscrete}, compactness and saturation imply that if there are finite $Y_t$ of unbounded size, then there is $t_0$ such that $Y_{t_0}$ is discrete with respect to inclusion. Let us see that this contradicts weak o-minimality. 

Let $\{y_i\}_{i\in \Nn}\sub Y_{t_0}$ be such that for all $i\in \Nn$, $y_{{i+1}}$ is an immediate successor in $Y_{t_0}$ of $y_{i}$.  By compactness, there is $\alpha>0$ in $M$,  such that for all $i$ there are $y_i<x_1<x_2<y_{i+1}$ with $x_2-x_1>\alpha$. 
For $y\in Y_{t_0}$ let $s(y)$ denote the successor of $y$ in $Y_{t_0}$. By our choice of $\alpha$, the set of  $y\in Y_{t_0}$, such that  there are $y<x_1<x_2<s(y)$ with $x_2-x_1>\alpha$ is infinite (since it contains $\{y_i\}_{i\in \Nn}$). 

For every code $y\in Y_{t_0}$, let $D_y$ be the convex subset of $M$ given by 
$\{x\in M: x<y< x+\alpha \}$. Since $y$ is non-valuational, $D_y\neq \0$. 
Then $\bigcup\{D_y: y\in Y_{t_0}\}$ is a definable set consisting of infinitely many convex components. This contradicts weak o-minimality.  
\end{proof}

\begin{lemma} \label{L: eliminating type (iii)}
	Assume that the collection of definable convex subgroups of $M$ has uniform finiteness. Then the collection of definable valuational initial segments has uniform finiteness.
\end{lemma}
\begin{proof}
	Assume that we have a definable family $\{Y_t:t\in T\}$, with each $Y_t$ consisting of (codes of) initial segments, $C\sub M$, with stabilizers $H=G(C)$, such that $\pi(C)\sub M/H$ is non-valuational.

	 For each $t\in T$, let $H(t)=\{G(C):C\in Y_t\}$. Then $H(t)$ is a set of convex subgroups of $(K,+)$ and  $|H(t)|\leq |Y_t|$. By the assumption, there is some $n\in \mathbb{N}$, such that $H(t)$ is finite if and only if $|H(t)|\leq n$. Thus, since we want to bound the size of the finite $Y_t$, we may assume that for every $t\in T$, $|H(t)|\leq n$. Since the groups in $H(t)$ are linearly ordered by inclusion, we may partition each $Y_t$ according to the associated group, and assume that for every $t\in T$, $|H(t)|=1$.

Assume now that the result fails.
    Then, there exists a strictly increasing sequence of natural numbers $n_k$, and for every $k$, some $t_k\in T$, such that $n_k=|Y_{t_k}|$. Thus,  there is some $t_0$ for which $Y_{t_0}$ is infinite and discretely ordered, and there is some fixed $H$ such that for all $C\in Y_{t_0}$, $G(C)=H$.
    
    Consider now the family of non-valuational initial segments inside $M/H$: $\{\pi(C):C\in Y_{t_0}\}$. Since each such $C$ consists of cosets of $H$, the map $C\mapsto C/H$ is injective and order preserving on $Y_{t_0}$, hence it contains a relatively convex subsequence of order type $\omega$.

	This in turn gives rise to a definable family of  sets of non-valuational cuts in the weakly o-minimal group (so dense) $M/H$, whose finite sets are unbounded in size, contradicting  Lemma \ref{L: eliminating non-valuational}. 
\end{proof}

    


\begin{theorem}\label{T: convex subgroups is enough}
    Let $\CM=(M,+,<,\dots)$ be a sufficiently saturated weakly o-minimal expansion of an ordered group. If the collection of all definable convex subgroups of $M$ has uniform finiteness then $\CM^{\eq}$ has uniform finiteness.
\end{theorem}
\begin{proof}
    By Lemma \ref{L: initial segments}, it suffices to show that the collection of definable initial segments of $M$ has uniform finiteness. 
    
	In Lemma \ref{L: definable sets in wom} we classified the different types of initial segments in $\CM$.  Given a definable family $\{Y_t: t\in T\}$  where each $Y_t$ is a set of (codes of) definable initial segments, we have to show that there is $n\in \Nn$ such that for all $t\in T$ such that $|Y_t|<\infty$ it follows that $|Y_t|< n$. For $t\in T$ partition $Y_t$ (definably, uniformly in $t$) into $Y_t^{(a)}$ and $Y_t^{(b)}$ according to whether $y\in Y_t$ is a code for an initial segment which is non-valuation (possibly rational) or valuational, respectively. It will suffice to show that there is $n$ such that for all $t$, if any one of $Y_t^{(a)}$, $Y_t^{(b)}$ is finite, it has size at most $n$.  
    The collection $\{Y_t^{(a)}: t\in T\}$ is handled using Lemma \ref{L: eliminating non-valuational} and $\{Y_t^{(b)}: t\in T\}$ is handled using Lemma \ref{L: eliminating type (iii)}.     
    \end{proof}

\begin{remark}\label{R: Thm ok for def wom}
    Although Theorem \ref{T: convex subgroups is enough} is stated for weakly o-minimal structures, the theorem also holds for weakly o-minimal definable sets in sufficiently saturated structures, i.e., definable ordered abelian groups $(D,+,<)$ for which every $\CM$-definable subset of $D$ is a finite union of convex definable sets. Namely, we do not need to assume that $D$ is stably embedded.

    More concretely: Let $D$ be a weakly o-minimal definable ordered abelian group in a sufficiently saturated structure $\CM$. If the collection of all definable convex subgroups of $D$ has uniform finiteness then every $\CM$-definable imaginary sort in $D$ has uniform finiteness.
\end{remark}

\subsection{Uniform finiteness in valued fields}

Let $\CK=(K,+,\cdot,<,\dots)$ be a sufficiently saturated weakly o-minimal expansion of an ordered field.\\

The following is straightforward.
\begin{lemma}\label{L:union of discrete}
    Let $(X,<)$ be a definable set with a definable linear order in some sufficiently saturated structure $\CM$ and $\{X_t\subseteq X:t\in T\}$ be a definable family.  Let $(t_n)_{n<\omega}$ be such that $X_{t_n}$ are finite and $|X_{t_n}|\to \infty$.

    For any definable subset $Y\subseteq X$,  there exists $t_0\in T$ such that either $X_{t_0}\cap Y$ or $X_{t_0}\cap (X\setminus Y)$ are infinite and discretely ordered. 
\end{lemma}

\begin{theorem}\label{T:Main, UF iff conv val rings}
    Let $\CK$ be as above. Then $\CK^{\eq}$ has uniform finiteness if and only if the collection of all definable convex valuation subrings of $\CK$ has uniform finiteness.
\end{theorem}
\begin{proof}
Obviously if $\CK^{\eq}$ has uniform finiteness then the collection of all definable convex valuation subrings has uniform finiteness. We prove the reverse implication. We first show the following: \\

\noindent $(\dagger)$ Uniform finiteness of the collection of definable convex  valuation rings implies that for every definable valuation $v$ on $K$, the collection of all definable convex subgroups of $\Gamma_v$ has uniform finiteness.\\

Indeed, by Theorem \ref{T: convex subgroups is enough} (and Remark \ref{R: Thm ok for def wom}) it suffices to show that the collection of definable convex subgroups of $\Gamma_v$ has uniform finiteness, and since for every convex subgroup of $\Gamma_v$ there corresponds a unique convex valuation subring of $K$ (the corresponding coarsening of $\CO_v$), the result follows.
    
    To prove the theorem, we use Theorem \ref{T: convex subgroups is enough} again. Assume that $\{Y_t:t\in T\}$ is a definable family, with each $Y_t$ a definable set of (codes of) definable (convex) subgroups of $(K,+)$. The collection of all convex subgroups is linearly ordered by inclusion, and since we want to bound the size of those $Y_t$ which are finite, we may assume that for all $t$, $Y_t$ is discretely ordered by inclusion.  

    We assume towards a contradiction that there is no $n$ which bounds the size of all finite $Y_t$. By saturation there exists $t_0$, such that $Y_{t_0}$ is infinite and by our assumption discretely ordered by inclusion.

    For each  $H\in \bigcup_{t\in T} Y_t$, let $R(H)=\{a\in K: aH\sub H\}$ be the corresponding (convex) valuation ring. The family $\{R(H):H\in \bigcup_{t\in T}Y_{t}\}$ is a uniformly definable chain (with respect to inclusion) of valuation rings  and therefore its intersection $\CO$ is a definable convex valuation ring. Let $\Gamma$ be the corresponding value group and $v:K^\times\to \Gamma$ the corresponding valuation.
	
	For every convex group $H$, let $v(H)=\{v(a):a\in H\}$, a final segment of $\Gamma$. Each such $v(H)$ either has a minimal element in $\Gamma$ or not. By Lemma \ref{L:union of discrete}, after possibly replacing $Y_{t_0}$, either all the elements of $Y_{t_0}$ have a minimum or none of them have. We claim that in each of the two cases $H\mapsto v(H)$ is injective. Assuming this, we get that $\{v(H):H\in Y_{t_0}\}$ is a discrete family of subsets of $\Gamma$,  contradicting $(\dagger)$ (and Theorem \ref{T: convex subgroups is enough}, Remark \ref{R: Thm ok for def wom}). So we prove injectivity in each of these cases.

	Assume first that no $v(H)$, $H\in Y_{t_0}$, has a minimum. Note that in that case we have $H=v^{-1}(v(H))$, for assume that $v(a)\in v(H)$ for some $a\in K$, which we may assume to be strictly positive.     Then there is $b\in H$, $b>0$, with  $v(b)<v(a)$. It follows that $a<b$, so by convexity $a\in H$. Thus, the map $H\mapsto v(H)$ is injective.

    Now assume that for all $H\in Y_{t_0}$,  $v(H)$ has a minimum, denoted by  $\gamma_H$.
	Note that if $\gamma_H=r$ (and recall $H$ is convex) then $B_{>r}(0)\subsetneq H\sub B_{\geq r}(0)$  and therefore $R(H)\subseteq \CO$. By minimality of $\CO$, $R(H)=\CO$ and hence $H=B_{\geq r}(0)$. Hence $H\mapsto v(H)$ is injective and we are done.
\end{proof}

\begin{remark}  
As in Remark \ref{R: Thm ok for def wom}, the theorem also holds for weakly o-minimal definable fields (which are not necessarily stably embedded). Namely, let $K$ be weakly o-minimal definable field in a sufficiently saturated structure $\CM$. If the collection of all definable convex valuation subrings of $K$ has uniform finiteness then every $\CM$-definable imaginary sort in $K$ has uniform finiteness.
\end{remark}

In \cite{HaHaPeSs} we proved that if $\CK$ is power bounded $T$-convex then $\CK^{\eq}$ has uniform finiteness. The following corollary generalizes this result:

\begin{corollary}\label{C: UF for wom valued fields}
The following are equivalent for $\CK$ as before:
\begin{enumerate}
        \item $\CK^{\eq}$ has uniform finiteness.
        
        \item For any definable valuation $v$ on $K$ (possibly trivial), the collection of all definable convex subgroups of $\Gamma_v$ and $\bk_v$ has uniform finiteness. 

        \item There exists a definable valuation $v$ on $K$ (possibly trivial) such that the collection of all definable convex subgroups of $\Gamma_v$ and $\bk_v$ has uniform finiteness. 

        \item For any definable valuation $v$ on $K$ (possibly trivial), the collection of all definable convex subgroups of $\Gamma_v$ has uniform finiteness.

        \item For any definable valuation $v$ on $K$ (possibly trivial), the collection of all definable convex subgroups of $\bk_v$ has uniform finiteness.

\end{enumerate}
\end{corollary}
\begin{proof}
    $(1)$ easily implies the rest, (2) easily implies (3) by taking the trivial valuation and similarly (5) implies (1). We show the rest.

    $(3)\implies (1)$. Let $v$ be the given valuation, with valuation ring $\CO_v$. If $v$ is trivial then the collection of definable convex subgroups of $K$ has uniform finiteness so we are done by Theorem \ref{T: convex subgroups is enough}. So assume that $v$ is not trivial. 
    
    We use Theorem \ref{T:Main, UF iff conv val rings} and show that the collection of definable convex valuation subrings has uniform finiteness. Let $\{X_t: t\in T\}$ be a definable family of definable convex valuation subrings and let $(t_n)_n$ be such that $|X_{t_n}|\to\infty$. For any $t\in T$, and $R\in X_t$ either $R\subseteq \CO_v$ or $\CO_v\subseteq R$. Thus we may assume that in addition  either $R\subseteq \CO_v$ for every $R\in X_{t_n}$ and natural number $n$ or the other inclusion for all $R\in X_{t_n}$ and natural number $n$.

    If $R\subseteq \CO_v$ for every such $R$, then the induced valuation rings on $\bk_v$ form a family of definable convex subgroups of $\bk_v$ contradicting  $(3)$. If $\CO_v\subseteq R$ for every such $R$ then the induced definable convex subgroups in $\Gamma_v$ contradict $(3)$.

    $(4)\implies (1)$. This is a repeat of the proof of Theorem \ref{T:Main, UF iff conv val rings}, see $(\dagger)$.
\end{proof}

We still do not know the answer to the question below.  Recall, first, that for an o-minimal theory $T$ in a signature $\CL$ expanding the theory of real closed fields, the $T$-convex expansion of $T$ (also denoted $T_{\mathrm{conv}}$) is the expansion of $T$ by a unary predicate for a convex valuation ring that is closed under all total $\emptyset$-$\mathcal{L}$-definable continuous unary functions.

\begin{question}    
Let $\CK$ be a T-convex valued field (not necessarily power-bounded). Does $\CK^{\eq}$ have uniform finiteness?
\end{question}

\appendix

\section{An example and a  resplendency theorem for weakly o-minimal structures}
In this section we give an example of a weakly o-minimal expansion of a valued field which does not have uniform finiteness in its imaginary sort. To build such an example, we start with a weakly o-minimal expansion $\Gamma$ of an ordered abelian group, admitting an infinite discrete imaginary sort (see, e.g., \cite[Example A.6]{HaHaPeSs}). We then consider the real closed valued field $\mathbb R((t^\Gamma))$ together with an expansion of RCVF by the extra structure on $\Gamma$, and show that it is weakly o-minimal. The key observation is a weakly o-minimal variant of known resplendence results for henselian valued fields. Explicitly, we show that  weakly o-minimal expansions of the value group and the residue field of a (pure) real closed valued field preserve weak o-minimality of the valued field sort.  


We first recall: Let $\CM$ be a $|T|^+$-saturated structure in a language $\CL$ and $T=\mathrm{Th}(M)$ its theory. Assume that $D$ is a stably embedded sort in $\CM$. Let $\mathcal{D}$ be $D$ with its $\emptyset$-induced structure and $\widetilde{\mathcal{D}}$ an arbitrary expansion of $\mathcal{D}$. It is folklore (see, e.g., \cite[Proposition 2.7]{CodEHaJiRi} for more details)  that letting $\widetilde L$, $\widetilde{\CM}$  be the corresponding expansions, then $D$ is stably embedded in $\widetilde{M}$ and the $\widetilde{L}$ induced structure on $D$ is $\widetilde{\mathcal{D}}$. Below we apply this without further mention to expansions of the RV-sort of (pure) real closed valued fields.

Let $\CK=(K,\cdot,+,0,1,\CO)$ be a sufficiently saturated model of RCVF. We denote by $\mathbf{RV}=(\mathrm{RV},\cdot,0,1,\oplus)$ the interpretable RV-sort, augmented by an     element $0$ for which we set $\rv(0)=0$ and the ternary relation \[
\oplus(x,y,z):=(\exists a,b)(\rv(a)=x\land  \rv(b)=y\land \rv(a+b)=z).\] 

\begin{fact}\label{F:oplus defined}
    For a fixed $x,y\in \mathrm{RV}$ the following are equivalent:
    \begin{enumerate}
        \item There is a unique $z\in \mathrm{RV}$ such that $\oplus(x,y,z)$.
        \item $x=0$ or $y=0$ or $x\neq \rv(-1)y$.
        \item $x=0$ or $y=0$ or $v(x+y)=\min\{v(x),v(y)\}$.
    \end{enumerate}
\end{fact}
\begin{proof}
    $(1)\iff (3)$ is \cite[Proposition 2.4]{flenner} (and the couple of lines after the proof).

    $(2)\iff (3)$ \cite[Proposition 2.2]{flenner}.
\end{proof}

The above allows us to define a partial function: 
\begin{notation}
        For $x, y\in \mathrm{RV}$ we write $x\oplus y$ for the  unique element $z\in \mathrm{RV}$ satisfying $\oplus(x,y,z)$, if such a $z$ is unique. If such a $z$ is not unique $x\oplus y$ is undefined. 
\end{notation}

It is well known that if $(K,v)$ is a henselian valued field of equicharacteristic $0$  then in the structure $\CK=(K,\mathbf{RV},\rv)$, every formula is equivalent to one without field-quantifiers (i.e., quantifiers ranging over variables from the valued field sort), \cite[Proposition 4.3]{flenner}. As a direct consequence, $\mathbf{RV}$ is stably embedded, see for example \cite[Remark A.10]{Ri:analytic}. These two results hold resplendently, i.e. 
they remain valid under arbitrary expansions of the  $\mathbf{RV}$-sort, see the discussion before \cite[Proposition 4.3]{flenner}. As a consequence:

\begin{fact}\label{F:fieldQE-RV}\cite[Proposition 5.1]{flenner}
    Let $\widetilde{\CK}$ be $\CK$ together with some added structure on $\mathbf{RV}$.     Any definable set $X\subseteq K$ is of the form $\{x\in K: (\rv(x-a_1),\dots, \rv(x-a_n))\in S\}$ for some definable subset $S\subseteq \mathrm{RV}^n$ and $a_1,\dots,a_n\in K$.
\end{fact}

Since $\CK$ is real closed, the order on $K$ is definable, and therefore so is the induced order on $\mathrm{RV}$.  Therefore, there is no harm expanding both structures by these orders. From now we let $\CK=(K,\cdot,+,0,1,<,\CO)$ and $\mathbf{RV}=(\mathrm{RV},\cdot,0,1,<,\oplus)$.


  \begin{lemma}\label{L:Rvwom implies K wom}
      Let $\widetilde \CK$ be $\CK$ together with a weakly o-minimal enrichment of $\mathbf{RV}$ which we denote by $\widetilde{\mathbf{RV}}$. Then $\widetilde \CK$ is still weakly o-minimal.
  \end{lemma}
\begin{proof}
    Let $X\subseteq K$ be a definable set in $\widetilde \CK$, by Fact \ref{F:fieldQE-RV} it is of the form \[\{x\in K: (\rv(x-a_1),\dots, \rv(x-a_n))\in S\},\] for some definable subset $S\subseteq \mathrm{RV}^n$.

    For any $1\leq i\leq n$ let $U_{i}=\{x\in K: \bigwedge_j v(x-a_{i})\geq v(x-a_j)\}$. This is a finite union of convex sets and $K=\bigcup_i U_i$.

    Let $D_{i}=\{y\in \mathrm{RV}: (y\oplus \rv(a_{i}-a_1),\dots, y\oplus \rv(a_{i}-a_n))\in S\}$. 
    Let us verify that  
    \[X=\bigcup_{i} U_{i}\cap \{x\in K: \rv(x-a_{i})\in D_{i}\}.\]

    Right to left inclusion. Let $x\in U_i$ with $\rv(x-a_i)\in D_i$. In particular, $\rv(x-a_i)\oplus \rv(a_i-a_j)$ is defined for all $j$ and thus is necessarily equal to $\rv(x-a_j)$. This gives that $(\rv(x-a_1),\dots, \rv(x-a_n))\in S$.
    
    Left to right inclusion. Let $x\in X$, then $x\in U_{i}$ for some $i$. To show that $\rv(x-a_{i})\in D_{i}$ it is enough to show that for all $j$, $\rv(x-a_{i})\oplus \rv(a_i-a_j)$ is defined, for then as above it is equal to $\rv(x-a_j)$ and the result follows since $x\in X.$ 
    
    We show that $\rv(x-a_{i})\oplus \rv(a_i-a_j)$ is defined using Fact \ref{F:oplus defined}. Assume that $\rv(x-a_{i})\oplus \rv(a_i-a_j)$ is undefined for some $j$: necessarily $x- a_{i}\neq 0$, $a_j-a_i\neq 0$ and $\rv(x-a_{i})=\rv(a_{j}-a_i)$. By \cite[Proposition 2.2]{flenner}, $v(a_j-a_i-(x-a_i))>v(x-a_i)$, i.e. $v(x-a_j)>v(x-a_{i})$ contradicting $x\in U_{i}$. 
    
    By weak o-minimality of $\widetilde{\mathbf{RV}}$, $D_{i}$ is a finite union of convex sets, thus $\{x\in K:\rv(x-a_{i})\in D_{i}\}$ is also a finite union of convex sets. We conclude that $X$ is a finite union of convex sets.
\end{proof}

    Transferring weak o-minimality from the $\mathrm{RV}$ sort to the valued field is the first step, the second is to transfer from $\Gamma$ and $\bk$ to $\mathrm{RV}$. For that we will need to deal with short exact sequences.
    The following is proved in \cite{AsChGeZi} but taken in this form from \cite[Fact 1.74]{Touchard}. Note that as $A$ is divisible, it is pure in $B$.

    \begin{fact}\label{F:relQE in short exact}
        Let $\CM=(0\to A\xrightarrow[]{\iota} B\xrightarrow{v}C\to 0)$ be a short exact sequence of abelian groups, with $A$ divisible, allowing expansion of the group structure on $A$ and $C$. 

        Then $A$ and $C$ are stably embedded in $\CM$ and every definable subset of $B$ is a boolean combination of definable sets of the following forms:
        \begin{enumerate}
            \item $\{x\in B: (v(t_1(x)),\dots, v(t_n(x))\in V\}$, where the $t_i$ are terms in the group language and $V$ is a definable subset of $C^n$.
            \item $\{x\in B: (\rho(t_1(x)),\dots, \rho(t_n(x))\in U\}$, where the $t_i$ are terms in the group language, $U$ is a definable subset of $A^n$ and $\rho(a)=\iota^{-1}(a)$ if $a\in \iota(A)$ and $0$, otherwise.
        \end{enumerate}
    \end{fact}

    Using this result we prove a general weak o-minimality transfer result.

    \begin{proposition}\label{P:wom shortexact}
    Let $\CM=(0\to A\xrightarrow[]{\iota} B\xrightarrow{v}C\to 0)$ be a short exact sequence of abelian groups, with $A$ divisible. 

        Let $\widetilde{\CM}=((B,+,0),\widetilde{A}=(A,+,0<,\dots),\widetilde{C}=(C,+,0,<,\dots),\iota,v)$ be $\CM$ together with weakly o-minimal expansions of $A$ and $C$. 
        \begin{enumerate}
            \item There is a unique (definable) order $<$ on $B$ such that with respect to this order, $\iota$ and $v$ are order preserving homomorphisms, $\iota(A)$ is a convex subgroup and $(B,+,<)$ is an ordered abelian group.
            \item The sort $(B,+,0,<)$ (together with its induced structure) is weakly o-minimal.
        \end{enumerate}
    \end{proposition}
    \begin{proof}
        (1) The order on $B$ defined by $b>0$ if and only if $v(b)>0$ or $v(b)=0$ and $\iota^{-1}(b)>0$, is the unique order on $B$ rendering it an ordered abelian group such that $\iota$ and $v$ are order preserving. 
        
        (2) Applying Fact \ref{F:relQE in short exact}, every definable subset of $B$ is a boolean combination of definable sets of the form 
        \begin{enumerate}
            \item $\{x\in B: (v(a_1+m_1x),\dots, v(a_k+m_kx))\in V\}$ for $a_i\in B$, $m_i$ an integer and $V$ a definable set in $\widetilde{C}$ and
            \item $\{x\in B:(\rho(a_1+m_1x),\dots,\rho(a_k+m_kx))\in U\}$ for $a_i\in B$, $m_i$ an integer and $U$ a definable set $\widetilde{A}$, where $\rho(a)=\iota^{-1}(a)$ if $a\in \iota(A)$ and $0\in A$ otherwise. 
        \end{enumerate}
        
        It suffices to show that each kind of these definable sets is a finite union of convex sets.
        
        Let $X\subseteq B$ be a definable set of the first kind. Written differently, it is the same as $v(x)\in \{y: (v(a_1)+m_1y,\dots, v(a_k)+m_ky)\in V\}$. The latter is a definable subset of $\widetilde{C}$ so a finite union of convex sets by weak o-minimality. As $v$ is order preserving, $X$ is a finite union of convex subsets of $B$.

        Let $X\subseteq B$ be a definable set of the second kind. For each $J\subseteq \{1,\dots,k\}$ let $P_J=\{x\in B: a_i+m_ix\in \iota(A)\iff i\in J\}$. Note that $a_i+m_ix\in \iota(A)\iff v(a_i)+m_iv(x)=0$. It is a finite covering of $B$ so it is enough to show that $X\cap P_J$ is a finite union of convex sets.

        \begin{claim}
            Each $P_J$ is a finite union of  convex sets.
        \end{claim}
        \begin{claimproof}
             Set $Y=\{y\in C:\bigwedge_{i\in J} m_iy=-v(a_i)\wedge \bigwedge_{i\notin J}m_iy\neq-v(a_i)\}$. This is a finite union of convex sets and since $P_J=\{x\in B:v(x)\in Y\}$, $P_J$ is a finite union of convex sets as well.      
        \end{claimproof}
        
        

        If $m_i=0$ for every $i\in J$ then $X\cap P_J=\{x\in P_J: (\rho(a_1'),\dots,\rho(a_k'))\in U\}$, where $a_i'=a_i$ if $i\in J$ and $a'_i=0$ otherwise. So $X\cap P_J$ is either empty or all of $P_J$, and thus either way a finite union of convex sets.

        Suppose there exists $i\in J$ for which $m_i\neq 0$. If $P_J=\emptyset$ then we are done. Otherwise, fix $b\in P_J$. Then for every $x\in P_J$, $v(a_i)+m_iv(x)=0=v(a_i)+m_iv(b)$ which implies that $v(x)=v(b)$. Thus, $P_J\subseteq v^{-1}(\gamma)$, where $\gamma:=v(b)\in C$. For every $i\in J$, let $c_i\in A$ be such that $\iota(c_i):=a_i+m_ib$.
        \begin{claim}
            $X\cap P_J=b+\iota(Y_J)$, for $Y_J=\{y\in A:(d_1(y),\dots,d_k(y))\in U\}$ where $d_i(y)=c_i+m_iy$ if $i\in J$ and $0$ otherwise.
        \end{claim}
        \begin{claimproof}
            Let $x\in X\cap P_J$ so $x=b+\iota(t)$ for some $t\in A$. We will show that $t\in Y_J$. Indeed, for every $i\in J$, $d_i(t)=c_i+m_it=\iota^{-1}(a_i+m_ib+m_i\iota(t))=\iota^{-1}(a_i+m_ix)=\rho(a_i+m_ix)$. The other direction is similar.
        \end{claimproof}

        As $Y_J$ is definable in $A$, it is a finite union of convex sets and thus so is $X\cap P_J$, as required.
    \end{proof}

        We wish to apply this result to a short exact sequence coming from $\mathrm{RV}, \bk$ and $\Gamma$.

    Let $\mathbf{RV}_{\bk,\Gamma}=((\mathrm{RV},\cdot, 0,1),\bk=(\bk,\cdot,+,0,1),\mathbf{\Gamma}=(\Gamma,+,0,<,\infty),\iota,v)$ be the three sorted structure associated  together with a $0$ element to $\bk$ and $\mathrm{RV}$ and an $\infty$ element to $\Gamma$. 
     Whenever $(K,v)$ is a henselian equicharacteristic $0$ valued field, both $\Gamma$ and $\bk$ are stably embedded inside $\mathbf{RV}_{\bk,\Gamma}$, as a pure ordered abelian group and as a pure field, respectively. As noted in \cite[\S 3]{ChSi:inp} the ternary relation $\oplus$ on $\mathrm{RV}$ is also definable in $\mathbf{RV}_{\bk, \Gamma}$.

    By Proposition \ref{P:wom shortexact}(1) the induced order on $\mathrm{RV}$ is definable.

    \begin{lemma}\label{L:wom shortexact of RV}
        Let $\widetilde{\mathbf{RV}_{\bk,\Gamma}}=((\mathrm{RV},\cdot, 0,1,<),\widetilde{\bk}=(\bk,\cdot,+,0,1,<,\dots),\widetilde{\mathbf{\Gamma}}=(\Gamma,+,0,<,\infty,\dots),\iota,v)$ be $\mathbf{RV}_{\bk,\Gamma}$ together with an enrichment of $\bk$ and $\Gamma$ still rendering them weakly o-minimal. Then the $\mathrm{RV}$ sort of $\widetilde{\mathbf{RV}_{\bk,\Gamma}}$ is  weakly o-minimal.
    \end{lemma}
    \begin{proof}
       Let $\mathrm{RV}_{>0}$ be $\rv(K_{>0})$. It is enough to show that $\mathrm{RV}_{>0}$ is weakly o-minimal (there is a definable bijection between $\mathrm{RV}_{>0}$ and $\mathrm{RV}_{<0}$).      Note that $\mathrm{RV}_{>0}$ is a subgroup of $\mathrm{RV}$. Also, $\bk_{>0}$ is a multiplicative subgroup and $\mathrm{RV}_{>0}/\bk_{>0}\cong \Gamma$.

         We have a short exact sequence $1\to \bk_{>0}\to \mathrm{RV}_{>0}\to \Gamma\to 0$, $\bk_{>0}$ is divisible because $\bk$ is real closed. 
         
        We are now in the situation of Proposition \ref{P:wom shortexact} where on $\Gamma$ we take the reverse order in order to make $v$ order preserving.
        \end{proof}

    \begin{theorem}\label{T:appendix theorem}
        Let $\CK$ be a pure real closed valued field and let $\widetilde{\CK}$ be  the structure $\CK$ together with weakly o-minimal expansions of $\bk$ and $\Gamma$. Then $\widetilde{\CK}$ is weakly o-minimal as well.
    \end{theorem}
    \begin{proof}
           The $\widetilde{\mathbf{RV}_{\bk,\Gamma}}$ reduct is weakly o-minimal by Lemma \ref{L:wom shortexact of RV}.  As $\oplus$ is definable in this structure (see  \cite[Section 3]{ChSi:inp}), $\widetilde{\mathbf{RV}}$ is a reduct of $\widetilde{\mathbf{RV}_{\bk,\Gamma}}$. By Lemma \ref{L:Rvwom implies K wom}, $\widetilde{K}$ is weakly o-minimal.
    \end{proof}

    \begin{corollary}
        There exists a weakly o-minimal valued field $\CK$ for which $\CK^{\eq}$ does not have uniform finiteness.
    \end{corollary}
    \begin{proof}
        In \cite[Example A.6]{HaHaPeSs}, we constructed a weakly o-minimal ordered abelian group $\mathcal{Q}=(Q,+,0,<)$ and a convex equivalence relation $E$ on $Q$ such that $\widetilde{\mathcal{Q}}=(Q,+,0,<,E)$ is still weakly o-minimal but $Q/E$ is discrete. There is no harm in assuming that $\mathcal{Q}$ (and $\widetilde{\mathcal{Q}}$) is $|T|^+$-saturated.
    
        Let $\CK$ be a $|T|^+$-saturated pure real closed valued field with value group $\mathcal{Q}$, e.g. $\mathcal{R}((t^{\mathcal{Q}}))$, for some $\mathbb{R}\prec \mathcal{R}$. By Theorem \ref{T:appendix theorem}, the structure $\widetilde{\mathcal{K}}$ obtained from $\mathcal{K}$ after expanding the value group to $\widetilde{\mathcal{Q}}$ is still weakly o-minimal. As $\widetilde{\mathcal{Q}}^{\eq}$ does not have uniform finiteness, $\widetilde{\mathcal{K}}^{\eq}$ does not have it as well.
    \end{proof}
    
\bibliographystyle{plain}
\bibliography{harvard}

\end{document}